\documentclass[12pt]{amsart}
\usepackage[a4paper,margin=1.15in, marginparwidth=2.5cm]{geometry}
\usepackage{amsmath,amssymb,amsthm}
\usepackage[T1]{fontenc}
\usepackage[utf8]{inputenc}

\usepackage{libertine}
\usepackage[libertine]{newtxmath}

\usepackage{tikz}
\usetikzlibrary{intersections, calc}

\usepackage{hyperref}
\hypersetup{
    colorlinks=true,      
    linkcolor=blue,       
    filecolor=magenta,    
    urlcolor=blue,        
    citecolor=blue,       
    pdftitle={T\'{\i}tulo de mi Art\'{\i}culo}, 
    pdfpagemode=FullScreen,
}

\newtheorem{theorem}{Theorem}
\newtheorem{claim}[theorem]{Claim}
\newtheorem{lemma}[theorem]{Lemma}
\newtheorem{cor}[theorem]{Corollary}
\newtheorem{prop}[theorem]{Proposition}

\theoremstyle{definition}
\newtheorem*{remark}{Remark}
\newtheorem*{notation}{Notation}
\newcommand{\R}{\mathbb{R}}
\newcommand{\HH}{\mathcal H}

\newcommand{\dist}{\operatorname{dist}}
\newcommand{\diam}{\operatorname{diam}}

\newcommand{\EE}{\mathcal{E}}
\newcommand{\BB}{\mathcal{B}}

\DeclareMathOperator{\card}{card}
\DeclareMathOperator{\supp}{supp}
\DeclareMathOperator{\length}{length}

\usepackage[noadjust]{marginnote}
\usepackage{ifoddpage}

\newcommand{\mystarn}  
{\checkoddpage\ifoddpage\reversemarginpar\else\normalmarginpar\fi\marginnote{\nr{$\pmb{\longrightarrow}$\quad}}}
\newcommand{\nr}[1]{\textcolor{red}{#1}}           

\title[Curly Kakeya]{Curly Kakeya: The Hausdorff dimension of sets\\ containing circles or line segments in many directions}
\author{A. C\'ordoba}
\address{Universidad Aut\'onoma de Madrid and ICMAT, Madrid, Spain}
\email{antonio.cordoba@uam.es}

\subjclass[2020]{42B25, 28A75} 
\keywords{Hausdorff dimension, Minkowski sausage}
\date{\today}

\begin{document}

\begin{abstract}
In this paper, we establish a lower bound for the Hausdorff dimension of a set $K\subset \mathbb{R}^n$ that contains a dilated and translated copy of every meridian,
that is, of every $(d-2)$-dimensional sphere passing through the poles of $S^{d-1}$, a $(d-1)$-dimensional sphere, with $3\le d\le n$, and $S^{d-1}$ contained in $\mathbb{R}^n$.
Under this assumption, we have
\[
\dim_H(K)\ge d-1.
\]
This result is closely related to, and reminiscent of, the classical Kakeya set problem. We build upon this connection, employing similar techniques to show
that if $K\subset \mathbb{R}^n$ contains a unit straight line segment in every direction corresponding to a smooth curve on $S^{n-1}$,
then its Hausdorff dimension is $\ge 2$.
\end{abstract}

\maketitle

\section{Circles}

To facilitate the reading, we shall consider first the relevant case $n=3$. There is a small difference in the final computation between this case and the general $n\ge4$, which probably does not justify completely their separate writing. Let us hope, nevertheless, that the redundancy is compensated by the helpful 
 familiarity with three dimensions and the fact that for $n=3$, the result  suggests an intriguing open problem for \mbox{$n>3$}. Furthermore, the proof can be easily extended to other surfaces of revolution with mild assumptions about the curvature of its generatrix. Dimension $n=2$ has its own peculiarities, and will be treated immediately after the case $n=3$ (Proposition~\ref{pr5}).

Given a circle $C$ in $\R^3$ as the boundary of a two-dimensional planar disc $D$, i.e., $C=\partial D$, we define the  direction of $C$ to be the normal direction to $D$.

Let us fix the line North-South in $\R^3$ passing through 
 the origin of coordinates and consider all directions $\theta$ perpendicular to it. That is, $\theta\in \EE$ (the equator) defines the direction of sets of parallel circles in  $\R^3$.

%
%

\begin{figure}[ht]

\begin{tikzpicture}[scale=3, line join=round, line cap=round]
  \def\R{1}     
  \def\rx{0.5}  
  \def\ry{0.3}  

  \tikzset{dashed lines/.style={color=gray, dashed, line width=0.6pt}}
  \tikzset{visible lines/.style={thick, line width=0.8pt}}

  \draw[visible lines] (0,0) circle (\R);

  \draw[dashed lines] (\R,0) arc (0:180:{\R} and \ry);
  \draw[visible lines, color=blue, name path=ecuador] (-\R,0) arc (180:360:{\R} and \ry);

  \draw[dashed lines] (0,\R) arc (90:-90:\rx\space and \R);
  \draw[visible lines, color=red, name path=meridiano] (0,-\R) arc (270:90:\rx\space and \R);

  \path[name intersections={of=ecuador and meridiano, by=X}];

  \draw[line width=2.5pt, color=red] (X) arc (195.23:150:\rx\space and \R);

  \fill[black] (0,0) circle (0.4pt);
  \node[above=2pt] at (0,\R) {$N$};
  \node[below=2pt] at (0,-\R) {$S$};
  \node[color=blue] at (0.8,-0.28) {$\mathcal{E}^+$};
  \node[color=black] at (-0.2,0.7) {$C_\theta$};
  \node[color=red, black] at (-0.4,0.1) {$C_\theta^1$};

\end{tikzpicture}

\caption{Octant $C^1_{\theta}$}\label{fig1}
\end{figure}


\begin{theorem}\label{th1} Let $K$ be a bounded subset of\/ $\R^3$ containing a circle $C_{\theta}$ of radius $r(\theta) >0$ in each direction $\theta\in \EE$. Then we have the lower bound
\[
\dim_H(K)\ge 2
\]
for the Hausdorff dimension of\/ $K$.
\end{theorem}

\begin{notation}
Given one of those circles $C_{\theta}$, the choice of North-South direction and horizontal (equatorial) plane allows a natural division of $C_{\theta}$ into eight octants, $C^j_{\theta}$, among which we choose one, $C^1_{\theta}$, not passing through the poles. See Figure~\ref{fig1}.
\end{notation}

\begin{remark} Let us recall that given a bounded set $E\subset \R^n$ and given $\alpha\ge0$ and $\delta>0$, we have the definition
\[
\HH_\delta^\alpha(E):=\inf\Big\{\sum_k [\diam(Q_k)]^\alpha\Big\},
\]
where the ``inf'' is taken over all countable coverings of $E$ by cubes $Q_k$ of diameter less than $\delta$. We define the Hausdorff measure as
\[
\HH^\alpha(E):=\lim_{\delta\to0}\HH_\delta^\alpha(E).
\]
The limit above always exists, although it can be infinite, and has the following properties:
\begin{enumerate}
\item $\HH^\alpha$ is countably additive, invariant under translations and rotations and satisfies
\[
\HH^\alpha(\lambda E)=\lambda^\alpha\HH^\alpha(E),\quad \lambda>0.
\]
\item $\HH^0(E)$ counts the number of points inside $E$, while
\[
\HH^n(E)=c_n|E|,
\]
where $|E|$ denotes Lebesgue measure in $\R^n$.
\item $\HH^\alpha(E)<\infty$ implies that $\HH^\beta(E)=0$ if $\beta>\alpha$. On the other hand, if $\beta<\alpha$ and $\HH^\alpha(E)>0$, then $\HH^\beta(E)=\infty$.
\end{enumerate}
Therefore, given a bounded  set $E\subset \R^n$, there is a unique $\alpha$ such that
\[
\HH^\beta(E)=0 \ \text{ if } \beta>\alpha,\quad \text{and}\quad \HH^\beta(E)=\infty \ \text{ if } \beta<\alpha.
\]
That unique $\alpha$ is called $\dim_H(E)$, the Hausdorff dimension  of $E$.
\end{remark}

An important observation is that, in the definition above, one can replace cubes by balls, or by cubes belonging to a mesh $\Gamma_k$ of width $2^{-k}$, where $k$ is a positive integer, without changing the value of $\dim_H(E)$.

There is another notion of dimension
in \(\mathbb{R}^n\), namely the box-counting  or
Minkowski dimension, \(\dim_B(F)\), of a bounded
set \(F\)  (see \cite{Fa}).

We have several equivalent definitions,
one of which is the following. Let
\(N_\varepsilon\) be the number of cubes in a
mesh of side length \(\varepsilon\) needed to
cover \(F\). Then
\[
\underline{\dim}_{\,B}(F)
=
\liminf_{\varepsilon \to 0}
\frac{\log N_\varepsilon}{|\log \varepsilon|}
\quad\text{and}\quad
\overline{\dim}_{B}(F)
=
\limsup_{\varepsilon \to 0}
\frac{\log N_\varepsilon}{|\log \varepsilon|}\cdot
\]
In the case that both limits above
coincide, their common value is
the box-counting dimension: \(\dim_B(F)\).

If we denote by \(F_\varepsilon=\{x: \operatorname{dist}(x,F)\leq \varepsilon\}\)
the Minkowski sausage  (following the terminology of Mandelbrot \cite{Ma}) of the set
\(F\), then \(\dim_B(F)\) describes the way
in which the measure \(|F_\varepsilon|\) decreases,
that is,
\[
\underline{\dim}_{\,B}(F)
=
n-
\limsup_{\varepsilon \to 0}
\Big|
\frac{\log |F_\varepsilon|}{\log \varepsilon}
\Big|
\quad\text{and}\quad
\overline{\dim}_{B}(F)
=
n-
\liminf_{\varepsilon \to 0}
\Big|
\frac{\log |F_\varepsilon|}{\log \varepsilon}
\Big|
\]

Also we have the inequality:
\[
\underline{\dim}_{\,B}(F)
\geq
\dim_{H}(F).
\]

Observe that since $\dim_H(S^2)=2$, the estimate of Theorem~\ref{th1} is best possible.

\begin{proof}[Proof of Theorem~\textup{\ref{th1}}] Let  $\{Q_\sigma\}$ be a covering of the set $K$ by dyadic cubes of side length
$
\ell(Q_\sigma)\le 2^{-k_0},
$ such that
\[
K\subset\bigcup_\sigma Q_\sigma=\bigcup_{k\ge k_0}\Big(\bigcup_\sigma Q_\sigma^k\Big),
\]
where $Q_\sigma^k\in\Gamma_k$, $\ell(Q_\sigma^k)=2^{-k}$. Defining
\[
E^k=\bigcup_\sigma Q_\sigma^k,
\quad
N_k=\card\{Q_\sigma^k\},
\]
the proof of Theorem~\ref{th1} follows if for every $\alpha<2$, there exists $\lambda_\alpha >0$ such that we have the estimate
\[
\sum_{k\ge k_0}N_k\,2^{-k\alpha}\ge\lambda_\alpha>0.
\]
But, in fact, it will be convenient to prove something stronger, namely:

\begin{claim}\label{cl2}
For every $\alpha<2$,  there exists $\lambda_\alpha >0$ such that, given a covering as above, there is $k\ge k_0$ such that
\[
N_k\,2^{-k\alpha}\ge\lambda_\alpha>0.
\]
\end{claim}

By hypothesis, $K$ contains a circle $C_\theta$ of radius $r(\theta)>0$ in each direction  $\theta\in\EE$. An elementary argument allows us to conclude the existence of $c>0$ such that the set $\EE^{\ast}=\{\theta\in\EE, c\leq r(\theta)\leq 2c\}$ has positive (exterior) measure $\HH^1(\EE^{\ast})>0$.

\begin{lemma}\label{le3}For each $\beta>0$, there exist $\lambda_\beta>0$ and $k\ge k_0$ such that
\[
  \HH^1(C_\theta^1\cap E^k)\ge\lambda_\beta\,2^{-k\beta},\quad\text{for each $\theta\in\EE^{\ast}$}.
\]
\end{lemma}

\begin{proof} Observe that
\[
\HH^1(C_\theta)=\sum_{k\ge k_0}\HH^1(C_\theta\cap E^k)
=\sum_{j=1}^8\sum_{k\ge k_0}\HH^1(C_\theta^j\cap E^k).
\]
The assumption that for each $k\ge k_0$ we have
\[
\HH^1(C_\theta^1\cap E^k)\le\lambda_\beta\,2^{-k\beta}
\]
yields the inequality
\[
\HH^1(C_\theta)\le\sum_{k\ge k_0}\lambda_\beta \,2^{-k\beta}+\sum_{k\ge k_0}\sum_{m>1}\HH^1(C_\theta^m\cap E^k)
\le \lambda_\beta\,\frac{2^{-k_0\beta}}{1-2^{-\beta}}+\frac78\,\HH^1(C_\theta),
\]
which is clearly false for $\lambda_\beta$ small enough. Therefore, for each  $\theta\in\EE^{\ast}$, we have
\[
\HH^1(C_\theta^1\cap E^k)\ge \lambda_\beta \,2^{-k\beta},\quad \text{for a certain } k\ge k_0.\qedhere
\]
\end{proof}

Next let us introduce the sets
\[
  \Omega_k=\{\theta\in\EE^{\ast} : \HH^1(C_\theta^1\cap E^k)\ge\lambda_\beta\,2^{-k\beta}\},
\]
 and observe that
\[
\mathcal{E}^{\ast}\subseteq \bigcup_{k\ge k_0} \Omega_k,
\]
  because for each $\theta\in\EE^{\ast}$, there exists $k\ge k_0$ satisfying $\theta\in\Omega_k$, and so
\[
 \HH^1(\EE^{\ast})\le\sum_{k\ge k_0}\HH^1(\Omega_k).
\]
It then follows that there exists $k\ge k_0$ such that
\[
\HH^1(\Omega_k)\ge\lambda_\beta\, 2^{-k\beta},
\]
for $\lambda_\beta$ small enough, implying that we can select $2^{k(1-\beta)}\lambda_\beta$ directions $\theta_j$ in $\Omega_k$ such that  $|\theta_j-\theta_\ell|\ge 2^{-k}$, for $j\ne \ell$, and thus, ordering properly,
\[
|\theta_j-\theta_\ell|\ge 2^{-k}|j-\ell|,
\quad j,\ell=1,\ldots,2^{k(1-\beta)}\lambda_\beta.
\]
  To each East-Northeast octant $C_{\theta_j}^1$ contained in $K$ in those selected directions, we associate the Minkowski sausages
\begin{align*}
D_{\theta_j}&=\{x:\dist(x,C_{\theta_j}^1)\le 2^{-k}\},
\\
\widetilde{D}_{\theta_j}&=\{x:\dist(x,C_{\theta_j}^1\cap E^k)\le 2^{-k}\},
\end{align*}
and observe that
\[
|D_{\theta_j}\cap E^k|\simeq |\widetilde{D}_{\theta_j}|.
\]
More precisely, there is a universal constant $a$, $1\leq a <\infty$, such that
\[
\frac{1}{a}\,|D_{\theta_j}\cap E^k|\le  |\widetilde{D}_{\theta_j}|\le {a}\,|D_{\theta_j}\cap E^k|.
\]
This is because
\[
D_{\theta_j}\cap E^k\subset \{x: M(\chi_{\widetilde{D}_{\theta_j}})\ge 1/4\}\quad\text{and}\quad
\widetilde{D}_{\theta_j}\subset \{x: M(\chi_{{D}_{\theta_j}\cap E^k})\ge 1/4\},
\]
where $M$ denotes the Hardy--Littlewood maximal function.

Since we have the estimate
\[|\widetilde{D}_{\theta_j}|\gtrsim \lambda_\beta\,2^{-k\beta}\,2^{-2k},
\]
we get the inequality
\begin{equation}
\sum_j |D_{\theta_j}\cap E^k|\gtrsim\lambda_\beta^2 \,2^{-k}\, 2^{-2k\beta}.
\label{eq1}
\end{equation}

The proof of Claim \ref{cl2} will be an immediate consequence of the estimate
\begin{equation}
\Big|\Big(\bigcup_j D_{\theta_j}\Big)\cap E^k\Big|
\gtrsim\frac{\lambda_\beta^3}{k}\,2^{-k}\,2^{-3k\beta}.
\label{eq2}
\end{equation}
This is because
\[
N_k\,2^{-3k}=|E^k|\ge\Big|E^k\cap\Big(\bigcup_jD_{\theta_j}\Big)\Big|
\gtrsim\frac{\lambda_\beta^3}{k}\,2^{-k}\,2^{-3k\beta},
\]
and therefore,
\[
N_k\,2^{-k\alpha}\gtrsim\frac{\lambda_\beta^3}{k}\,2^{k(2-\alpha-3\beta)}\ge\lambda_\alpha>0
\]
for $\alpha<2$, so long as $\beta>0$ has been chosen small enough.

\vskip1pt
Finally, to prove \eqref{eq2}, we write
\[
\sum_j |D_{\theta_j}\cap E^k|=\int\Big[\sum_j\chi_{D_{\theta_j}\cap E^k}\Big].
\]
Then \eqref{eq1}, together with H\"{o}lder's inequality, yields
\begin{align*}
\lambda_\beta^2 \,2^{-k(1+2\beta)}
&\le\Big|\Big(\bigcup_jD_{\theta_j}\Big)\cap E^k\Big|^{1/2}\,\Big[ \sum_{i,j} |D_{\theta_i}\cap D_{\theta_j}|\Big]^{1/2}
\le|E^k|^{1/2}\,\Big[\sum_{i,j}\frac{2^{-2k}}{|i-j|+1}\Big]^{1/2}.
\end{align*}
The estimation of $|D_{\theta_1}\cap D_{\theta_j}|$ follows from a straightforward trigonometrical calculation, taking into account that the dihedral angle between the meridians $C_{\theta_i}$ and $C_{\theta_j}$ has value $\simeq |i-j|\, 2^{-k}$, see
Figure \ref{fig2}.
\begin{figure}[ht]
\begin{tikzpicture}[scale=3, line join=round, line cap=round, >=stealth]
  \def\R{1}     
  \def\rx{0.5}  
  \def\ry{0.3}  
  \def\rxB{0.65}  

  \tikzset{dashed lines/.style={color=gray, dashed, line width=0.6pt}}
  \tikzset{visible lines/.style={thick, line width=0.8pt}}

  \draw[visible lines] (0,0) circle (\R);

  \draw[dashed lines] (\R,0) arc (0:180:{\R} and \ry);
  \draw[visible lines, color=blue, name path=ecuador] (-\R,0) arc (180:360:{\R} and \ry);

  \draw[visible lines, color=red, name path=meridiano] (0,-\R) arc (270:90:\rx\space and \R);

  \draw[visible lines, color=orange, name path=meridianoB] (0,-\R) arc (-90:90:\rxB\space and \R);

  \draw[dashed lines] (0,\R) -- (0,-\R);

  \path[name intersections={of=ecuador and meridiano, by=X}];
  \draw[line width=2.5pt, color=red] (X) arc (195.23:150:\rx\space and \R);

  \path[name intersections={of=ecuador and meridianoB, by=XB}];
  \draw[line width=2.5pt, color=orange] (XB) arc (-13.44:30:\rxB\space and \R);

  \draw[line width=0.4pt, dashed] (0,0) -- (X);
  \draw[line width=0.4pt, dashed] (0,0) -- (XB);

  \draw[line width=0.4pt] (-0.11,-0.06) arc (205:346:0.14 and 0.11);

  \fill[black] (0,0) circle (0.4pt);
  \node[above=2pt] at (0,\R) {$N$};
  \node[below=2pt] at (0,-\R) {$S$};

  \node[color=red, left] at (-0.45,0.2) {$D_{\theta_j}$};
  \node[color=orange, right] at (0.6,0.3) {$D_{\theta_i}$};

  \node[scale=0.8] at (0.04,-0.18) {\footnotesize$\theta_j-\theta_i$};


  \draw[->, line width=0.6pt] (1.3,0.1) to[out=20,in=160] (1.7,0.1);

  \begin{scope}[xshift=1.35cm, yshift=-0.1cm, scale=1.2]

    \draw[line width=2.5pt, color=orange] ({\rxB*cos(-13.44)},{\R*sin(-13.44)}) arc (-13.44:30:\rxB\space and \R);

    \begin{scope}[xshift=1.05cm, yshift=-0.08cm, rotate=-15]
      \draw[line width=2.5pt, color=red] ({\rx*cos(195.23)},{\R*sin(195.23)}) arc (195.23:150:\rx\space and \R);

    \begin{scope}[xshift=-0.78cm, yshift=-0.05cm, rotate=23]
    \draw[line width=2.5pt, color=black] (0.355,0.10) -- (0.375,0.13);
\end{scope}
\end{scope}
  \end{scope}

\end{tikzpicture}

\caption{Intersection of $D_{\theta_j}$ and $D_{\theta_i}$.}\label{fig2}
\end{figure}

  Since
\[
\sum_{i,j=1}^T\frac{1}{|i-j|+1}\le T\sup_i\sum_{j=1}^T\frac{1}{|i-j|+1}\lesssim T\log T,
\]
and in our case we have
\[
T\simeq 2^{k(1-\beta)}\, \lambda_\beta,
\]
  we get
\[
|E^k|\gtrsim\frac{\lambda_\beta^3}{k}\,2^{-k(1+3\beta)}.
\qedhere\]
\end{proof}

Theorem~\ref{th1} is sharp, as the sphere $S^2\subset\mathbb R^3$ shows. It has the following corollary, in the spirit of the classical Kakeya needle problem.

\begin{cor}\label{co4} A bounded set  $K\subset\mathbb R^3$ containing a  circle in every spatial position must satisfy
\[
\dim_H(K)\ge 2.
\]
\end{cor}

The method of the proof allows the extension of Corollary~\ref{co4} to other rectifiable plane curves $\Gamma$ under mild curvature conditions: a set containing a dilated copy of $\Gamma$ in every spatial position has Hausdorff dimension $\ge 2$.
But the non-rectifiable case, like for example the Von Koch snowflake, remains open.

One may naturally ask the same question in higher dimensions $(n\ge4)$, but the methods developed in this paper produce only the lower bound $\dim_H(K)\ge2$ for a set $K\subset\mathbb R^n$ containing a  circle in general position, while one expects the sharper estimate $\dim_H(K)\ge n-1$. The suprematistic approach finds here the same kind of difficulty that is confronted with in the Kakeya case, namely, the control of the overlapping of parallelepipeds in the metric of $L^p(\mathbb R^n)$, when $1<p<2$.


We can then observe that with this generality, Corollary~\ref{co4} fails in dimension $n=2$, for the following reason. Let $C$ be a Cantor set of points $x\in[0,1)$ such that the expansion $x=0,x_1x_2\cdots$ on base $4$ contains only the digits $\{0,3\}$. That is, we divide the interval $[0,1)$ into four subintervals,
\[
[0,1)=[0,1/4)\cup[1/4,1/2)\cup[1/2,3/4)\cup[3/4,1),
\]
remove the two placed at the center, and then repeat the same process with the other two and so on.

Let \[\tilde  C=\frac12 \,C=\Big\{\frac12\, x: x\in C\Big\},\] and place copies of $C$ and $\tilde  C$ at, respectively, the horizontal lines $y=0$ and $y=1$:
\[
\tilde C\subset \tilde E=\{(x,y):0\le x\le 1/2;\ y=1\}
\quad\text{and}\quad
C\subset E=\{(x,y):0\le x\le1;\ y=0\}.
\]
Let $B$ (a Besicovitch set) be the set of straight line segments joining a point of the top~$\tilde  C$ with another at the bottom $C$. We have (see \cite{Gu,Ka,St}) that $B$ is compact, has measure zero and contains a segment in each direction obtained joining a point of $\tilde  E$ with a point in~$E$. That is, covering angles (directions) $\theta$ in the interval $-1\le \tan(\theta)\le 2$. Therefore, the union of four conveniently rotated copies $B_j$ of $B$ is an example of a Kakeya--Besicovitch set in $\mathbb R^2$. If we denote by $C_j$ and $\tilde  C_j$ the corresponding Cantor sets, their union
\[
K:=\Big(\bigcup_{j=1}^4 C_j\Big)\cup\Big(\bigcup_{j=1}^4 \tilde  C_j\Big)
\]
has Hausdorff dimension $1/2$, and inside $K$ one can find, for each plane direction $\theta$, two points $z(\theta)$ and $ w(\theta)$ such that
\[
z(\theta)-w(\theta)=r(\theta)\theta,\quad \text{for some } r(\theta),\ 1\le r(\theta)\le2.
\]

In the other direction, we have the following.

\begin{prop}\label{pr5}
Let $K\subset \mathbb{R}^{2}$ be a bounded
set containing, for each $\theta\in S^{1}$, two points
$z(\theta)$ and $w(\theta)$ such that $z(\theta)-w(\theta)=r(\theta)\cdot\theta$, for
some $r(\theta)$, $1\leq r(\theta)\leq 2$. Then
\[\underline{\dim}_{\,B}(K)\geq \frac{1}{2}\cdot\]
\end{prop}

\begin{proof}
Let $E_{\varepsilon}=\{Q_{\sigma}\}$ be a covering of $K$
by squares of a mesh $\Gamma_{\varepsilon}$ in $\mathbb{R}^{2}$ of
width $\varepsilon$, i.e., $\ell(Q_{\sigma})=\varepsilon$, $K\subset\bigcup Q_{\sigma}$,
and define $N_{\varepsilon}=\operatorname{card}(E_{\varepsilon})$.

\vskip1pt
First we fix a set of directions
\[
A_{\varepsilon}=\{\theta_{j}\}_{j=1,2,\dots,[\varepsilon^{-1}]}\subset S^{1}_{+}
\]
so that
\[
|\theta_{j}-\theta_{k}|\simeq \varepsilon |j-k|.
\]

Let us also define
\[
z(\theta_{j})\in Q^{z}(\theta_{j})\in E_{\varepsilon}\quad\text{and}\quad
w(\theta_{j})\in Q^{w}(\theta_{j})\in E_{\varepsilon},
\]
 with
\[
z(\theta_{j})-w(\theta_{j})=r(\theta_{j})\cdot\theta_{j},\quad
1\leq r(\theta_{j})\leq 2.
\]
Let
\begin{align*}
D_{\theta_{j}}&=Q^{z}(\theta_{j})\cup Q^{w}(\theta_{j}),
\\
\mathcal{F}&= \{ \{Q^{z}(\theta_{j}),Q^{w}(\theta_{j}) \}:1\leq j\leq [\varepsilon^{-1}] \}.
\\
\mathcal{R}(Q)&= \{Q_{\sigma}\in E_{\varepsilon}:\{Q,Q_{\sigma}\}\in\mathcal{F} \}.
\\
M(Q)&=\operatorname{Card}\{\mathcal{R}(Q)\}=\text{number of directions }
\theta_{j}\in A_{\varepsilon}\text{ so that }D(\theta_{j})\supset Q.
\end{align*}

Next let us choose $Q_{1}$ such that $M(Q_{1})$ is maximum. Having chosen
$Q_{1},\dots,Q_{\tau}$, we pick~$Q_{\tau +1}$ with the same criterion,
namely, $M(Q_{\tau +1})$ is maximum in
$E_{\varepsilon}\setminus \{Q_{1},\dots,Q_{\tau}\}$. We continue this
selection process until all the directions in $A_{\varepsilon}$ are obtained
within the sets
\[
\{Q_{1},\dots,Q_{T}\}\cup\bigcup_{j=1}^{T}\{Q_{\sigma}:Q_{\sigma}\in\mathcal{R}(Q_{j})\}.
\]
Obviously, we have
\[
\sum M(Q_{j})=[\varepsilon^{-1}],
\]
and from this we conclude the existence of
$\nu\in \{1,2,\dots,[|\log\varepsilon|]\}$ such that
\[
\varepsilon^{-1}\geq M_{\nu}\,2^{\nu}\geq \frac{1}{|\log\varepsilon|}\,\varepsilon^{-1},
\]
where
\[
M_{\nu}=\operatorname{Card}\{Q_{j}:2^{\nu-1}\leq M(Q_{j})<2^{\nu}\}.
\]

Let us denote by $\tilde {Q}_{1},\dots,\tilde {Q}_{M_{\nu}}$ those
squares in the family $\{Q_{1},\dots,Q_{T}\}$ satisfying the estimate
\[
2^{\nu-1}\leq \operatorname{Card}(\mathcal{R}(\tilde {Q}_{j}))<2^{\nu}.
\]

We have
\[
\sum_{j=1}^{M_{\nu}}
\Big(\sum_{Q\in\mathcal{R}(\tilde {Q}_{j})}|Q|\Big)
\gtrsim M_{\nu}\,2^{\nu}\varepsilon^{2}
\gtrsim \frac{\varepsilon}{|\log\varepsilon|},
\]
and
\begin{align*}
\sum_{j=1}^{M_{\nu}}
\Big(\sum_{Q\in\mathcal{R}(\tilde {Q}_{j})}|Q|\Big)
&=
\int\Big[
\sum_{j=1}^{M_{\nu}}
\sum_{Q\in\mathcal{R}(\tilde {Q}_{j})}
\chi_{Q}
\Big]
\leq |E_{\varepsilon}|^{1/2}\cdot
\Big[
\sum_{j=1}^{M_{\nu}}
\sum_{k=1}^{M_{\nu}}
\int
\sum_{\substack{Q^{1}\in\mathcal{R}(\tilde {Q}_{j})\\ Q^{2}\in\mathcal{R}(\tilde {Q}_{k})}}
\chi_{Q^{1}}\chi_{Q^{2}}
\Big]^{1/2}.
\end{align*}

Observe that for each $Q^{1}\in\mathcal{R}(\tilde {Q}_{j})$, there are at most
$2^{\nu+1}$ squares in $ \bigcup_{k=1}^{M_{\nu}}\mathcal{R}(\tilde {Q}_{k})$
intersecting (in fact coinciding) with $Q^{1}$, because otherwise $Q^{1}$
will have been one of the selected $\tilde {Q}_{j}$. Therefore we obtain
for the integral above the estimate
\[
\lesssim M_{\nu}\,2^{2\nu}\varepsilon^{2},
\]
Since we also have the estimate
\[
\lesssim M_{\nu}^{2}\,2^{\nu}\varepsilon^{2}
\]
for that integral, taking the geometric mean, we get
\[
\frac{\varepsilon}{|\log\varepsilon|}
\lesssim |E_{\varepsilon}|^{1/2}
 \,[(M_{\nu}\,2^{\nu})^{3/2}\varepsilon^{2} ]^{1/2}
\]
and we use now the fact $M_{\nu} 2^{\nu}\leq \varepsilon^{-1}$ to obtain
\[
|E_{\varepsilon}|\gtrsim \frac{\varepsilon^{3/2}}{|\log\varepsilon|^{2}},
\]
implying that
\[
\underline{\dim}_{\,B}(K)
\geq 2-\limsup_{\varepsilon\to 0}
\Big|\frac{\log |E_{\varepsilon}|}{\log\varepsilon}\Big|
\geq \frac12\cdot\qedhere
\]
\end{proof}

It remains, nevertheless, to decide the intriguing question about the dimension of $K\subset\mathbb R^2$ when we fix the distance $r(\theta)=1$. That is, when the set of differences $K-K$ contains a circle of radius $1$.

Assuming that $s=\dim_H(K)=\dim_B(K)=$ box counting dimension, we have
\[
\min(2s,2)\ge \dim_H(K-K)\ge \max(s,1),
\]
where the lower bound is obvious, while the upper bound comes from the fact that the map $(x,y)\to x-y$ is Lipschitz from $K\times K\subset\mathbb R^4$ into $\mathbb R^2$, implying that
\[
\dim_H(K-K)\le \dim_H(K\times K)\le 2s,
\]
from which we deduce that $s\ge1/2$, but one may expect in this case the bound $s\geq 1$.

There is, nevertheless, an extension of Theorem~\ref{th1} to the case $n\ge4$ which is also sharp. Given a sphere $S^{d-1}$ of radius $1$ in the Euclidean space $\R^n$, $n> d\geq 2$, where we have fixed the poles $N$ and $S$, let us call ``meridian'' the $S^{d-2}$-sphere obtained as the intersection of~$S^{d-1}$ with a $(d-1)$-plane $H$ passing by the poles, and assume that we have a coordinate system in~$\R^n$ such that
\[
S^{d-1}=\Big\{(x_1,\ldots,x_d,0,\ldots,0):\sum_{j=1}^d x_j^2=1\Big\},
\]
\[
N=(0,\ldots,1,0,\ldots,0),\quad
S=(0,\ldots,-1,0,\ldots,0),
\]
\[
\EE=S^{d-1}\cap\{x_d=0\},
\quad
\EE^+=\EE\cap\{x_1\ge0\}.
\]
Each meridian passing by the poles is the intersection of a $(d-1)$-plane $H$ with $S^{d-1}$ and it has a unique normal $\nu\in\EE^+$. Therefore we can parametrize that meridian family by those normal vectors: $\{C_\nu\}_{\nu\in\EE^+}$.

\begin{theorem}\label{th6}Given $n\ge4$, $3\le d\le n$, let $K\subset\R^n$ be a bounded set containing a translated and dilated copy of every $(d-2)$-sphere meridian passing by the fixed poles of a $S^{d-1}$ sphere. Then the Hausdorff dimension of\/ $K$ must satisfy
\[
\dim_H(K)\ge d-1.
\]
\end{theorem}

\textit{Proof.} Let $\{Q_\sigma\}$ be a covering of $K$ by dyadic cubes of side length
$
\ell(Q_\sigma)\le2^{-k_0}
$ such that
\[
K\subset\bigcup_\sigma Q_\sigma=\bigcup_{k\ge k_0}\Big(\bigcup_\sigma Q_\sigma^k\Big),
\]
where $Q_\sigma^k\in\Gamma_k$, $\ell(Q_\sigma^k)=2^{-k}$ and $\Gamma_k$ is a mesh in $\R^n$ of width $2^{-k}$.

Defining $E^k=\bigcup_\sigma Q_\sigma^k$, $N_k=\card\{Q_\sigma^k\}$ as in Theorem~\ref{th1}, the estimate for the Hausdorff dimension $(\dim_H K)$ will be a consequence of the following fact:

For every  $\alpha<d-1$, there exists $\lambda_\alpha >0$ such that for every covering as above there exists  $k\ge k_0$ such that $N_k\,2^{-k\alpha}\ge\lambda_\alpha>0$.

As in the proof of Theorem~\ref{th1}, we need a preliminary argument to reduce the general case to the particular situation where the radius
 is between two fixed values $c\leq r\leq 2c$. In order to simplify the presentation, in the following we shall assume that we have $1\leq r \leq 2$, leaving 
 the reader to complete the needed details.

Let us define
\[
C_\nu^1=C_\nu\cap\Big\{0\le x_d\le\frac12\Big\}
\]
and observe that
\[
\HH^{d-2}(C_\nu^1)= c_d\,\HH^{d-2}(C_\nu),
\]
where $0<c_d<1$ is a universal constant.

By hypothesis, $K$ contains translated copies $\tilde  C_\nu$ and $\tilde  C_\nu^1$ of $C_\nu$ and $C_\nu^1$, respectively.

\begin{lemma}\label{le7}For each $\beta>0$, there exist $\lambda_\beta>0$ and $k\ge k_0$ such that
\[
\HH^{d-2}(\tilde  C_\nu^1\cap E^k)\ge\lambda_\beta\,2^{-k\beta}.
\]
\end{lemma}

\textit{Proof.} We have
\[
\HH^{d-2}(\tilde  C_\nu)\le
\sum_{k\ge k_0}\HH^{d-2}(\tilde  C_\nu^1\cap E^k)
+\sum_{k\ge k_0}\HH^{d-2}\bigl((\tilde  C_\nu-\tilde  C_\nu^1)\cap E^k\bigr).
\]
Therefore the hypothesis
\[
\HH^{d-2}(\tilde  C_\nu^1\cap E^k)\le\lambda_\beta\,2^{-k\beta},
\quad\text{for every } k\ge k_0,
\]
implies the inequality
\[
\HH^{d-2}(\tilde  C_\nu)\le
\lambda_\beta\,\frac{2^{-k_0\beta}}{1-2^{-\beta}}+(1-c_d)\,\HH^{d-2}(\tilde  C_\nu),
\]
which is obviously false for $\lambda_\beta$ small enough.

\vskip1pt
Next, for each $k\ge k_0$, let us define
\[
\Omega_k=\{\nu\in\EE^+:\HH^{d-2}(\tilde  C_\nu^1\cap E^k)\ge\lambda_\beta\,2^{-k\beta}\},
\]
and observe that Lemma~\ref{le7} implies that
\[
\EE^+\subset\bigcup_{k\ge k_0}\Omega_k,
\]
because for every $\nu\in\EE^+$, there is $k\ge k_0$ satisfying $\nu\in\Omega_k$. Therefore
\[
\HH^{d-2}(\EE^+)\le\sum_{k\ge k_0}\HH^{d-2}(\Omega_k),
\]
and this implies that for a certain $k\ge k_0$, we must have
\[
\HH^{d-2}(\Omega_k)\ge\lambda_\beta\,2^{-k\beta}
\]
for a convenient $\lambda_\beta$, that we can always take to be smaller than the previously chosen constant $\lambda_\beta$.

From that estimate we can infer that $\Omega_k$ contains a set of points $A=\{\nu\}$ of cardinal
\[
\simeq \lambda_\beta\,2^{-k\beta}\,2^{k(d-2)}
\]
separated by a distance bigger than $2^{-k}$. That is, if $\nu$ and $\mu$ are different elements of $A$, we have
\[
\|\nu-\mu\|\ge2^{-k}\quad\text{or}\quad \|2^k\nu-2^k\mu\|\ge1.
\]
To each selected direction $\nu\in A$, we associate the Minkowski sausage in $\R^n$,
\[
D_\nu=\{x:\dist(x,\tilde  C_\nu^1)\le2^{-k}\}.
\]
Then we have the estimate
\[
\sum_{\nu\in A}\HH^n(D_\nu\cap E^k)
\gtrsim \lambda_\beta^2\, 2^{-k\beta}\,2^{-k(n-2d+4)}.
\]
And observe that
\begin{align*}
\sum_{\nu\in A}\HH^n(D_\nu\cap E^k)
&=\int\Big[\sum_{\nu\in A}\chi_{D_\nu\cap E^k}\Big]
\le\Big[\HH^n\Big(\Big(\bigcup_\nu D_\nu\Big)\cap E^k\Big)\Big]^{1/2}
\cdot\Big[\sum_{\nu,\mu\in A}\HH^n(D_\nu\cap D_\mu)\Big]^{1/2}
\end{align*}
by H\"{o}lder's inequality.

We may now assume here that $n\ge4$, because the case $n=3$, slightly different, was covered by Theorem~\ref{th1}. Elementary geometrical considerations yield the estimate
\[
\sum_{\nu,\mu\in A}\HH^n(D_\nu\cap D_\mu)
\lesssim\sum_{\nu,\mu\in A}\frac{2^{-k(n-d+2)}}{|2^k\nu-2^k\mu|+1}
\le \card(A)\sup_\mu\sum_\nu\frac{2^{-k(n-d+2)}}{|2^k\nu-2^k\mu|+1}.
\]
Since the distance between two points in $A$ is always greater than $2^{-k}$, using polar coordinates in $\R^{d-2}$, we can estimate the sum above by the integral
\[
\int_0^M r^{d-4}\,dr,
\quad \text{where } M=[\lambda_\beta\,2^{-k\beta}]^{1/(d-2)}\,2^k,
\]
and obtain the estimate
\[
\HH^n\Big( \Big(\bigcup D_\nu \Big)\cap E^k \Big)
\gtrsim \lambda_\beta^2 \,2^{-3k\beta}\,2^{-k(n-d+1)},
\]
implying
\[
N_k\,2^{-k\alpha}\gtrsim \lambda_\beta^2\, 2^{k(d-1-\alpha-3\beta)}\gtrsim\lambda_\alpha>0
\]
if $\beta$ is chosen small enough.

\section{Segments}

In $\R^n$, $n\ge2$, we have two types of Kakeya maximal functions, namely there is one where the averages are taken over ``needles'' or tubes, and another where the geometry is that of ``coins''. There are, of course, the intermediate cases, but here we will consider only needles or coins. More precisely, given a number $N>1$, the eccentricity, we have \mbox{$N$-needles} (tubes or cylinders of radius $r$ and height $rN$) and $N$-coins (cylinders of height~$h$ and radius $hN/2$), but in both cases their orientation will be allowed to be arbitrary in~$\R^n$.

\begin{figure}[ht]\label{fig:needle-coin}

\centering\resizebox{9cm}{!}{
\begin{tikzpicture}[line join=round, line cap=round]

  \begin{scope}[rotate=-35]
    \draw[dashed, black!50, thick] (0.15,0) arc (0:180:0.15cm and 0.06cm);

    \fill[gray!10, opacity=0.6] (-0.15,0) -- (-0.15,8) arc (180:360:0.15cm and 0.06cm) -- (0.15,0) arc (360:180:0.15cm and 0.06cm);

    \draw[thick] (-0.15,0) arc (180:360:0.15cm and 0.06cm);

    \draw[thick] (-0.15,0) -- (-0.15,8) (0.15,0) -- (0.15,8);

    \node at (0.5,4.1) {\footnotesize $rN$};
    \node at (0.35,-0.05) {\footnotesize$ r$};

    \draw[thick] (0,0) -- (0.15,0);

    \draw[thick, fill=gray!30] (0,8) ellipse (0.15cm and 0.06cm);
  \end{scope}

  \begin{scope}[xshift=8cm, yshift=2cm, rotate=-25]
    \draw[dashed, black!50, thick] (3.2,0) arc (0:180:3.2cm and 1.15cm);

    \fill[gray!10, opacity=0.6] (-3.2,0) -- (-3.2,0.3) arc (180:360:3.2cm and 1.15cm) -- (3.2,0) arc (360:180:3.2cm and 1.15cm);

    \draw[thick] (-3.2,0) arc (180:360:3.2cm and 1.15cm);

    \draw[thick] (-3.2,0) -- (-3.2,0.3) (3.2,0) -- (3.2,0.3);

    \draw[thick, fill=gray!30] (0,0.3) ellipse (3.2cm and 1.15cm);

    \draw[<->, >=stealth, thick] (3.45,-0.05) -- (3.45,0.35)
      node[midway, right, black] {\footnotesize $h$};

    \draw[thick] (0,0.3) -- (3.2,0.3)
      node[midway, above, black] {\footnotesize $hN/2$};
  \end{scope}

\end{tikzpicture}

}
\caption{A needle and a coin.}
\end{figure}

Needles or coins provide differentiation bases, $\BB_N^{\textup{needle}}$ and $\BB_N^{\textup{coin}}$, which are invariant under translations, dilations and rotations, and have associated maximal functions
\[
M_N^{\textup{needle}}f(x):=\sup \frac{1}{\mu(R)}\int_R |f(y)|\,dy,
\]
where the sup is taken over all elements $R$ of $\BB_N^{\textup{needle}}$ containing $x$, and
\[
M_N^{\textup{coin}}f(x):=\sup \frac{1}{\mu(R)}\int_R |f(y)|\,dy,
\]
where the sup is taken now over all elements of $\BB_N^{\textup{coin}}$ containing $x$.

Obviously, in dimension $n=2$ we have
\[
M_N^{\textup{coin}}=M_N^{\textup{needle}},
\]
but for $n\ge3$ they have very different behavior in term of $L^p$-boundedness. The main result of reference \cite{Co1} is the following.

\begin{theorem}\label{th8}For every $n\ge2$, there exist finite constants $C(n)$, $a(n)$, such that
\[
\|M_N^{\textup{coin}}f\|_{L^2(\R^n)}
\le C(n)[\log(N+1)]^{a(n)}\|f\|_{L^2(\R^n)}.
\]
\end{theorem}

Regarding the highly interesting $M_N^{\textup{needle}}$, the conjecture is that the logarithmic bound holds in the space $L^p(\R^n)$, $p\ge n$. This 
 is a result having consequences for the precise estimate of the Hausdorff dimension of Kakeya sets, but also for the understanding of the spherical summation operators (\cite{Co1} and \cite{Fe}), but it is only known for $n=2$ and remains open for $n\ge3$ where we only have partial results.

There is, however, an estimate valid in any dimension, so long as the set of directions of the needles is restricted to be in a smooth curve in $S^{n-1}$.

Let \[\gamma\colon [0,1]\to S^{n-1}\] be a $C^1$-curve satisfying the following property: there exists $C=C(\gamma)<\infty$ such that
\[
\HH^{n-1}\{\xi\in S^{n-1}:\card\{t\in[0,1]:\gamma'(t)\cdot\xi=0\}>C\}=0.
\]
Let us define
\[
\BB_N^\gamma=\{\text{needles }R,\ \text{eccentricity}(R)=N,\ \text{direction}(R)\in\gamma([0,1])\},
\]
and let us consider $M_N^\gamma$, the associated maximal function. Reference \cite{Co3} contains a proof of the following.

\begin{theorem}\label{th9}We have that
\[
\|M_N^\gamma f\|_{L^2(\R^n)}\le C(\gamma)[\log(N+1)]^{a(n)}\|f\|_{L^2(\R^n)},
\]
where $a(n)$ is a finite constant depending only upon the dimension $n$.
\end{theorem}

Theorem~\ref{th9} has the following application.

\begin{theorem}\label{th10}Suppose that $K_\gamma$ is a bounded set in~$\R^n$ containing a unit segment in each direction of the curve $\gamma$. Then
\[
\dim_H(K_\gamma)\ge2.
\]
\end{theorem}

\begin{proof} As in Theorem~\ref{th1}, let $\{Q_\nu\}$ be a covering of $K_\gamma$ by dyadic cubes of side length
\[
\ell(Q_\nu)\le2^{-k_0}.
\]
That is,
\[
K_\gamma\subset\bigcup_\nu Q_\nu=\bigcup_{k\ge k_0}\Big(\bigcup_\nu Q_\nu^k\Big),
\]
where $\ell(Q_\nu^k)=2^{-k}$. Define
\[
E^k=\bigcup_\nu Q_\nu^k,
\quad
N_k=\card\{Q_\nu^k\}=\text{ Number of cubes of length }2^{-k}\text{ in the covering}.
\]
It follows from the definition of Hausdorff dimension that Theorem~\ref{th10} will follow from the estimate
\[
\sum_{k\ge k_0}N_k\,2^{-k\alpha}\ge\lambda_\alpha>0
\]
for every $\alpha<2$. But, as in Theorem~\ref{th1}, it will be convenient to prove something stronger, namely: for every $\alpha<2$, there exists $k\ge k_0$ such that
\[
N_k\,2^{-k\alpha}\ge\lambda_\alpha>0.
\]

By definition, $K_\gamma$ contains a straight line segment $L_\theta$ of length $1$, for each direction $\theta$ in the curve $\gamma$. Therefore,
\[
1=\HH^1(L_\theta)\le\sum_{k\ge k_0}\HH^1(L_\theta\cap E^k),
\]
where $\HH^1$ denotes $1$-dimensional Hausdorff measure.

Arguing by contradiction, we may conclude that for each $\beta>0$, there exist $\lambda_\beta>0$ and $k\ge k_0$ such that
\[
\HH^1(L_\theta\cap E^k)\ge\lambda_\beta\,2^{-k\beta}.
\]
Next, we introduce the sets
\[
\Gamma_k=\{\omega\in\gamma([0,1]):\HH^1(L_\omega\cap E^k)\ge\lambda_\beta\,2^{-k\beta}\}
\]
and observe that
\[
\length(\gamma)\lesssim\sum_{k\ge k_0}\HH^1(\Gamma_k),
\]
because for each direction $\omega\in\operatorname{Im}(\gamma)$ there is, at least, a value $k\ge k_0$ satisfying
\[
\HH^1(L_\omega\cap E^k)\ge\lambda_\beta\,2^{-k\beta}.
\]
We can then infer the existence of $k\ge k_0$ so that
\[
\HH^1(\Gamma_k)\ge\lambda_\beta\,2^{-k\beta}
\]
(because in the opposite case we would 
 obtain the inequality
\[
\length(\gamma)\le\lambda_\beta\sum_{k\ge k_0} 2^{-k\beta}=\lambda_\beta\,\frac{2^{-k_0\beta}}{1-2^{-\beta}}
\]
which is obviously false for $\lambda_\beta$ small enough).

A consequence of the estimate $\HH^1(\Gamma_k)\ge\lambda_\beta\,2^{-k\beta}$ is that we can select $\lambda_\beta\,2^{k(1-\beta)}$ different angles $\{\omega_j\}$ inside $\Gamma_k$ so that
\[
|\omega_j-\omega_\ell|\gtrsim2^{-k}|j-\ell|.
\]
Consider now the corresponding $L_{\omega_j}$ segments contained in $K_\gamma$ in those selected directions, and associate to each of them a tube $R_j$ of radius $2^{-k}$ and height $=1$, pointing in the direction $\omega_j$.
We have
\[
\sum_j\mu(R_j)\gtrsim2^{-(n-1)k}\lambda_\beta\,2^{k(1-\beta)}.
\]
On the other hand, H\"{o}lder's inequality yields:
\begin{align*}
\sum_j\mu(R_j)=\int\sum_j\chi_{R_j}
&\le \Big[\mu\Big(\bigcup R_j\Big)\Big]^{1/2}\Big[\sum_{i,j}\mu(R_i\cap R_j)\Big]^{1/2}
\\
&\lesssim \Big[\mu\Big(\bigcup R_j\Big)\Big]^{1/2}
\Big[2^{-k(n-1)}\lambda_\beta\,2^{k(1-\beta)}\sup_i\Big(\sum_j\frac1{|i-j|+1}\Big)\Big]^{1/2},
\end{align*}
implying
\[
\mu\Big(\bigcup R_j\Big)\gtrsim \frac1k\,\lambda_\beta\,2^{-k(n-3+\beta)}.
\]
Finally, let us observe that
\[
\bigcup R_j\subset\{M_{2^k}^\gamma\chi_{E^k}\ge\lambda_\beta\,2^{-k\beta}\}.
\]
Therefore, using the estimate of Theorem~\ref{th9}, we get
\[
\mu(E^k)\gtrsim\frac1{k^{a(n)}}\,2^{-k[n-2+3\beta]}\,\lambda_\beta^3.
\]
Since
\[
N_k\,2^{-kn}=\mu(E^k),
\]
we get
\[
N_k\,2^{-k\alpha}\gtrsim\frac1{k^{a(n)}}\,\lambda_\beta^3 \,2^{k[2-\alpha-3\beta]}\ge\lambda_\alpha>0
\]
for every $\alpha<2$, so long as we have chosen $\beta>0$ small enough.
\end{proof}

\begin{remark}The estimate $\dim_H(K_\gamma)\ge2$ is sharp, and it suggests the following conjecture:
\begin{quotation}
\emph{Let $M$ be a $k$-dimensional submanifold of\/ $S^{n-1}$ and consider a bounded set $K_M\subset\R^n$ containing a unit length segment in each direction of\/ $M$. Then $\dim_H(K_M)\ge k+1$}.
\end{quotation}

In the important particular case $n=3$, $M=S^2$, the conjecture has been proven recently by Hong Wang and Joshua Zahl \cite{WZ}.
\end{remark}

The hypothesis in the theorems above can be relaxed asking only the inclusion of circles or intervals in a dense subset of directions. But the situation changes completely if we substitute spheres by balls, as the following theorem shows.

\begin{theorem}\label{th11}Let $K\subset\R^n$ be a bounded set containing a unit $(n-1)$-ball in each spatial direction. Then we have
\[
\dim_H(K)=n.
\]
\end{theorem}

\begin{proof} It will be a consequence of Theorem~\ref{th8} and of the method developed in Theorem~\ref{th1}. Using the same notation about the covering by dyadic cubes $\{Q_\nu\}$, the sets $E^k=\bigcup_\nu Q_\nu^k$, $k\ge k_0$, and arguing as in Theorem~\ref{th1}, we conclude the existence of $k\ge k_0$ and a set $A=\{\nu\}\subset S^{n-1}$ of directions containing $\lambda_\beta\,2^{(n-1)k}\,2^{-k\beta}$ elements, such that the $(n-1)$-dimensional ball $D_\nu\subset K$, having $\nu$ as normal direction, satisfies the estimate
\[
\HH^{n-1}(D_\nu\cap E^k)\ge\lambda_\beta\,2^{-k\beta},
\]
and
\[
\|\nu-\mu\|\ge2^{-k}\quad\text{for every }\nu\ne\mu\text{ in }A.
\]
Here, $\beta$ is any positive number and $\lambda_\beta$ is small enough.

Next, to each $D_\nu$, we associate the coin $\tilde  D_\nu$ of height $2^{-k}$ (radius $1$) and observe that
\[
\sum_\nu |\tilde  D_\nu|\gtrsim\lambda_\beta\,2^{-k\beta}\,2^{k(n-2)}.
\]
H\"{o}lder's inequality yields
\begin{align*}
\sum_\nu |\tilde  D_\nu|&=\int\sum_\nu\chi_{\tilde  D_\nu}
\le\Big|\bigcup_\nu\tilde  D_\nu\,\Big|^{1/2}
\Big[\sum_{\nu,\mu}|\tilde  D_\nu\cap\tilde  D_\mu|\Big]^{1/2}
\\
&\lesssim \Big|\bigcup_\nu\tilde  D_\nu\,\Big|^{1/2}
\Big[\lambda_\beta\,2^{k(n-1)}\,2^{-k\beta}\sup_\nu\Big(\sum_\mu\frac{2^{-k}}{\|\nu-\mu\|+1}\Big)\Big]^{1/2},
\end{align*}
from which we obtain
\[
\Big|\bigcup\tilde  D_\nu\,\Big|\ge\lambda_\beta\,2^{-k\beta}.
\]
Next, we observe that
\[
\bigcup\tilde  D_\nu\subset\{x:M_{2^k}^{\textup{coin}}(\chi_{E^k})\ge\lambda_\beta\,2^{-k\beta}\}
\]
and the estimate $N_k\,2^{-k\alpha}\gtrsim\lambda_\alpha>0$, for every $\alpha<n$, follows from Theorem~\ref{th8}, choosing $\beta>0$ small enough.
\end{proof}

\section{Appendix} For the sake of completeness, we will sketch now 
 the proofs of Theorems \ref{th9} and \ref{th8} (see~\cite{Co1,Co3}); 
 at least the parts of those proofs which were relevant in the discussion of Theorems~\ref{th9} and~\ref{th10} above.

\subsection{Sketch of the proof of Theorem~\ref{th9}} We first prove the estimate when the radius~$r$ has been fixed but, of course, the tubes of eccentricity $N$ are allowed to have arbitrary directions in the given curve $\gamma$. Also let us observe that by a dilation argument, it will be enough to consider the particular case $r=1$.

Next, we use an auxiliary function $\varphi\ge0$, satisfying $\varphi\in C_0^\infty(\R)$, $\supp(\varphi)\subset[-2,+2]$, $\varphi(x)=1$ for $|x|\le1$ and with Fourier transform, $\widehat\varphi$, real valued, to define the average operators
\[
T_{j/2^k}f(x)=\int f\Big(x-t\gamma\Big(\frac{j}{2^k}\Big)\Big)\,\varphi(t)\,dt.
\]
Here, $j$ is an integer, $0\le j\le2^k$.

Taking the Fourier transform, we get the expression
\[
\widehat{T_{j/2^k}f}(\xi)=\widehat f(\xi)\,\widehat\varphi\Big(\xi\cdot\gamma\Big(\frac{j}{2^k}\Big)\Big).
\]
Let us define
\[
M_{2^k}f(x)=\sup_j T_{j/2^k}f(x).
\]
We claim that
\[
\|M_{2^k} f\|_{L^2(\R^n)}\le C(\gamma)\,k  \|f\|_{L^2(\R^n)},
\]
which is proved by induction, because
\[
\|M_{2^{k+1}}f-M_{2^k}f\|
\le\Big[\sum_j |T_{(j+1)/{2^{k+1}}}f-T_{{j}/{2^{k+1}}}f |^2\Big]^{1/2}
=G_{2^k}f.
\]
and
\begin{align*}
\int |G_{2^k}f|^2&=\int |\widehat{G_{2^k}f}|^2
\lesssim\int|\widehat f(\xi)|^2\sum_{j=1}^{2^k}\Big|
\widehat\varphi\Big(\xi\cdot\gamma\Big(\frac{j+1}{2^{k+1}}\Big)\Big)
-\widehat\varphi\Big(\xi\cdot\gamma\Big(\frac{j}{2^{k+1}}\Big)\Big)\Big|^2\,d\xi.
\end{align*}
and because of the hypothesis about $\gamma$, we have
\[
\sum_{j=1}^{2^k}\Big|
\widehat\varphi\Big(\xi\cdot\gamma\Big(\frac{j+1}{2^{k+1}}\Big)\Big)
-\widehat\varphi\Big(\xi\cdot\gamma\Big(\frac{j}{2^{k+1}}\Big)\Big)\Big|^2\le C(\gamma)
\]
for almost every $\xi\in\R^n$.

To finish the proof, we use a pigeonhole argument, see \cite{Co1}, based on the fact that without loss of generality, one can consider only dyadic radius $2^{-k}$, and observe that if two needles~$R_1$ and $R_2$ of eccentricity $N$ have a non-empty intersection, $R_1\cap R_2\ne\emptyset$, and if $\operatorname{radius}(R_1)\le N\operatorname{radius}(R_2)$, then $R_1$ is included in the double $R_2^*$ of $R_2$.

\subsection{Sketch of the proof of Theorem~\ref{th8}}

Elementary considerations yield that, without loss of generality, we may consider only coins whose normal vectors make an angle less than $\pi/4$ with a fixed direction, which in the following will be taken as the vertical one.

First we fix the height $h$ and prove an estimate for the maximal function independent of $h$, which by a dilation argument can be taken to be $h=1$. Then a ``pigeonholing'' argument, similar to the one used in Theorem~\ref{th9}, will allow patching together the estimates for different values of $h$.

Taking into account that the diameter of the coins is $N$, it will be enough to consider the action of the maximal function over functions supported on a cube of side length equal to $N$.

Given such a cube $Q\subset\R^n$, we divide it into $N^n$ subcubes of side length $1$:
\[
Q=\bigcup_{\nu,j}Q_\nu^j,
\]
where $\nu=(\nu_1,\ldots,\nu_{n-1})$, $1\le\nu_i\le N$, denotes the horizontal position, while $j=1,\ldots,N$ indicates the vertical location of $Q_\nu^j$ within the column
\[
E_\nu=\bigcup_{j=1}^N Q_\nu^j.
\]
Then for each choice of coins $R_\nu^j\in\BB_N^{\textup{coin}}$ such that $Q_\nu^j\cap R_\nu^j\ne\emptyset$, we consider the linear operator
\[
Tf(x)=\sum_{\nu,j}\frac1{N^{n-1}}\int_{R_\nu^j}f(y)\,dy\cdot\chi_{Q_\nu^j}(x),
\]
whose adjoint is given by
\[
T^*f(x)=\sum_{\nu,j}\frac1{N^{n-1}}\int_{Q_\nu^j}f(y)\,dy\cdot\chi_{R_\nu^j}(x).
\]
Given $f\in L^2(Q)$, we have the decomposition
\[
f=\sum_\nu f_\nu,
\quad f_\nu=f|_{E_\nu},
\]
with
\begin{align*}
\int|T^*f_\nu|^2
&=
\sum_{j,k=1}^N\frac1{N^{2(n-1)}}
\int_{Q_\nu^j}f_\nu\int_{Q_\nu^k}f_\nu\,|R_\nu^j\cap R_\nu^k|^{1/2}
\\
&\lesssim\frac1{N^{2(n-1)}}\sum_{j,k}
\Big(\int_{Q_\nu^j}f_\nu^2\Big)^{1/2}
\Big(\int_{Q_\nu^k}f_\nu^2\Big)^{1/2}
\frac{N^{n-1}}{|j-k|+1}
\lesssim\frac{\log(N+1)}{N^{n-1}}\,\|f_\nu\|_2^2
\end{align*}
and
\[
\|T^*f\|_2\lesssim\sum_\nu\|T^*f_\nu\|_2
\lesssim\Big[\frac{\log(N+1)}{N^{n-1}}\Big]^{1/2}\sum_\nu\|f_\nu\|_2
\lesssim[\log(N+1)]^{1/2}\,\|f\|_2.
\]


\end{document}